\documentclass[11pt]{article}
\usepackage{latexsym, amssymb, amsmath, amsfonts, amsthm,graphics,
  combinedgraphics}

\begin{document}

\theoremstyle{plain}

\newtheorem{prop}{Proposition}
\newtheorem{theorem}{Theorem}
\newtheorem{corollary}{Corollary}
\newtheorem{lemma}{Lemma}

\theoremstyle{definition}
\newtheorem{definition}{Definition}
\newtheorem{remark}{Remark}
\newtheorem{remarks}{Remarks}
\newtheorem{example}{Example}
\newtheorem{examples}{Examples}

\newcommand{\rr}{\mathbb R}
\newcommand{\cc}{\mathbb C}
\newcommand{\nn}{\mathbb N}
\newcommand{\zz}{\mathbb Z}
\renewcommand{\ss}{\mathbb S}
\newcommand{\ts}{\thinspace}
\newcommand{\sse}{\subseteq}

\newcommand{\al}{\alpha}
\newcommand{\be}{\beta}
\newcommand{\ga}{\gamma}
\newcommand{\la}{\lambda}
\newcommand{\all}{\frac{\alpha}{|\alpha|}}
\newcommand{\ep}{\epsilon}
\newcommand{\alo}{\alpha_1}
\newcommand{\ai}{\alpha_i}
\newcommand{\ak}{\alpha_k}
\newcommand{\an}{\alpha_n}
\newcommand{\supp}{{\rm supp}}

\newcommand{\sig}{\sigma}
\newcommand{\D}{\Delta}
\newcommand{\Dn}{\Delta_n}
\newcommand{\pa}{\parallel}
\newcommand{\ta}{\tilde a}
\newcommand{\puis}{\rr((x^{\frac{i}{n}}))}
\newcommand{\cpuis}{\cc((x^{\frac{i}{n}}))}

\title{Polynomials non-negative on strips and half-strips}

\author{Ha Nguyen \footnote{
Department of Mathematics,
Wesleyan College,
Macon, GA 31210. Email:\,hnguyen@wesleyancollege.edu .}
\and
Victoria Powers \footnote{
Department of Mathematics and Computer Science,
Emory University,
Atlanta, GA 30322. Email:\, vicki@mathcs.emory.edu.} }

\maketitle

\section{Introduction}

Throughout, we work in the real polynomial ring in two variables, which we
denote by $\rr[x,y]$.  The set of sums of squares in $\rr[x,y]$ is denoted
by $\sum \rr[x,y]^2$.  
Recently, M.~Marshall \cite{MM_strip} settled a long-standing open problem by
proving the following:

\begin{theorem} \label{mar} Suppose $f(x,y) \in \rr[x,y]$ is non-negative
on the strip $[0,1] \times \rr$.  Then there exist $\sigma(x,y), \tau(x,y)
\in \sum \rr[x,y]^2$ such that $$f(x,y) = \sigma(x,y) + \tau(x,y) (x - x^2).$$
\end{theorem}

An expression $f = \sig + \tau (x - x^2)$ is an immediate witness to the
positivity condition on $f$.  In general, one wants to characterize polynomials
$f$ 
which are positive, or non-negative,  on a semialgebraic set $K \sse \rr^n$
in terms of sums of squares and the polynomials used to define $K$.  
Representation theorems of this type have a long and illustrious history, going
back at least to Hilbert.  There has been much interest in these questions
in the last decade, in a large part because of applications outside of
real algebraic geometry, notably in problems of optimizing polynomial
functions on semialgebraic sets.
In this paper we look at some generalizations of Marshall's theorem.  Our results
give many new examples of non-compact semialgebraic sets in $\rr^2$ for
which one can characterize all polynomials which are non-negative on the set.

Let $\rr[X] $ denote
$\rr[x_1, \dots , x_n]$, the real polynomial ring in $n$ variables,
and write $\sum \rr[X]^2$ for the sums of squares in $\rr[X]$.
Given a finite set $S = \{ s_1, \dots , s_k \} \sse \rr[X]$ 
the {\it basic closed semialgebraic set in $\rr^n$ generated by $S$}, denoted
$K_S$, is
$\{ a \in \rr^n \mid s_i(a) \geq 0$ for $i = 1, \dots , k  \}$.  
Note that the
strip $[0,1] \times \rr$ is the basic closed semialgebraic set in $\rr^2$
generated by $\{ x - x^2 \}$. 

There are two
algebraic objects associated to the semialgebraic set $K_S$:  The {\it quadratic module generated by
$S$}, denoted $M_S$,  is the set of all elements of $\rr[X]$ which can be written
$\sig_0 + \sig_1 s_1  + \dots + \sig_k s_k$, where each $\sig_i \in \sum \rr[X]^2$.
The {\it preordering } generated by $S$, denoted $T_S$,  consists
of all elements of the form $\sum_{e \in \{0,1\}^k} \sigma_e s^e$, where
$s^e$ denotes $s_1^{e_1} \dots s_s^{e_s}$ for $e = (e_1, \dots , e_s)$, and
each $\sigma_e \in \sum \rr[X]^2$.   In general, $M_S \subsetneqq T_S$, although
if $|S| = 1$, then clearly $T_S = M_S$. Also, $T_S = M_S$ iff $M_S$ is closed
under multiplication iff $s_i \cdot s_j \in M_S$ for all $i,j$.

We recall briefly what is known about the existence of representations in $T_S$ or $M_S$
for polynomials positive or non-negative on $K_S$.    If $K_S$ is compact, then Schm\"udgen's Theorem \cite{Schm} says
that every $f$ which is strictly positive
on $K_S$ is in $T_S$, regardless of the choice of generators $S$.  However, in general,
one cannot replace $f > 0$ on $K_S$ by $f \geq 0$  on $K_S$, or replace $T_S$ by
$M_S$.  If $K_S$ is not compact and $\dim(K_S) \geq 3$  then by \cite[Prop. 6.1]{S_99},
there always exist polynomials $f$ which are positive on $K_S$, but not in $T_S$, regardless
of the choice of generators $S$.   The same is true if $\dim(K_S) = 2$ and
$K_S$ contains an open cone, by \cite[Prop. 3.7]{PS} .  
By \cite[Thm. 2.2]{KM_MP}, if  $K_S \sse \rr$ and is not compact, then $T_S$ contains every $f$ which is
non-negative on $K_S$,  provided one chooses the right set of generators $S$.  If $K_S \sse \rr$ and
is compact, then $M_S$  contains all polynomials non-negative on $K_S$, again provided
one chooses the right set of generators.  
(We give an elementary proof of this in \S 2.)

We say that $M_S$ (respectively, $T_S$) is {\it saturated} if for every $f \in \rr[X]$,
$f$ non-negative on $K_S$ implies $f \in M_S$ (respectively, in $T_S$).  
Marshall's Theorem says that the
 the quadratic module in $\rr^2$ generated by $x - x^2$   is saturated.
This was only the second example given of  a finitely generated saturated
preordering in the non-compact case (the first being 
the preordering generated by $x, 1-x$ and $ 1 - xy$ given in  \cite[Rem. 3.14]{S}),
and settled a long-standing open problem.

Our aim in this paper is to give families of examples related to Marshall's
theorem.   In the next section, we generalize Marshall's result to the case
$U \times \rr$, where $U$ is any compact set in $\rr$, more precisely, we show
that if  $S \sse \rr[x]$ is the ``obvious" set of generators for $U$, then the quadratic
module in $\rr[x,y]$ generated by $S$ is saturated.   In \S 3, we look at some non-compact
subsets of
a strip $[a,b] \times R$ which are bounded as $y \rightarrow -\infty$; we refer to
such a set as a half-strip in $\rr^2$.  We give a representation theorem for a half-strip 
of the form $(U \times \rr) \cap \{ y \geq q(x) \}$, where $U \sse \rr$ is compact
and $q(x) \geq 0$ on $U$.  We give other examples of half-strips for which the
corresponding preordering is saturated, as well as a family of negative examples.

The authors are grateful to Bruce Reznick, and especially Murray Marshall, for helpful
discussions concerning the work in this paper.

\section{Polynomials non-negative on  strips in the plane}

In this section, we give representation theorems for non-compact basic closed  semialgebraic sets 
which are contained in a subset of $\rr^2$ of the form $[a,b] \times \rr$ and are
unbounded as $y \rightarrow \pm \infty$.  We refer to such a set as a {\it strip} in the plane.
We start with a representation theorem for strips of 
the form $U \times \rr$, where $U \sse \rr$ is compact.  More precisely, we show that
the quadratic module  corresponding to $U \times \rr$ is saturated, as long as we choose the
right set of generators.   We end this section with a few remarks about the
more general case of $U \times W$, where $W \sse \rr$ is a non-compact basic closed
semialgebraic set. 

For the rest of this section, fix $U \sse \rr$ compact, say $U = [a_1,b_1] \cup \dots \cup [a_k,b_k]$, where $a_1 \leq b_1 < a_2 \leq b_2 < \dots <a_k \leq b_k$.   Define $S \sse \rr[x]$ by
$$S = \{x - a_1, (x - a_2)(x - b_1), \dots, (x - a_k)(x - b_{k-1}), b_k - x \}.$$  Then the basic closed
semialgebraic set generated by $S$ in $\rr$ (respectively in $\rr^2$) is $U$ (respectively $U \times \rr$).  By analogy with the non-compact case
in $\rr$ (see \cite{KM_MP}), we call $S$ the {\it natural choice of generators} for $U$.  

\begin{lemma} \label{dim1} Suppose $U \sse \rr$ is compact and $S \sse \rr[x]$ is the natural choice
of generators.  Then in  $\rr[x]$, $T_S = M_S$.  It follows immediately that the same is
true in $\rr[x,y]$.
\end{lemma}

\begin{proof}  Let $U$ be as above  and, for ease of exposition, set
$s_i = (x - a_{i+1})(x - b_i)$ for $i = 1, \dots k-1$.
The identities
$$(x - a_1)(b_k - x) = \frac{1}{b_k - a_1} \left( (b_k - x)^2 (x - a_1) + (x - a_1)^2 (b_k - x) \right),$$
$$x - a_1 = \frac{1}{b_k - a_1} \left( (x-a_1)^2 + (x - a_1)(b_k - x)\right), \text{ and }$$
$$b_k - x = \frac{1}{b_k - a_1} \left( (b_k - x)^2 +  (x - a_1)(b_k - x)\right)$$
show that the quadratic module generated by $\{ (x - a_1)(b_k - x), s_1, \dots , s_{k-1} \}$ is the
same as $M_S$.  Thus 
to prove $T_S = M_S$, it is enough to prove that $s_i s_j \in M_S$ for $1 \leq i < j \leq k-1$
and $(x - a_1)(b_k - x) s_i \in M_S$ for $1 \leq i \leq k-1$.

 Suppose $1 \leq i < j \leq k-1$, then it is easy to check that
$s_i + s_j \geq 0$ on $[a_1,b_k]$.  Let $m$ be the maximum of $s_i + s_j$ on $[a_1,b_k]$,
then $s_i s_j \geq \frac{1}{m} (s_i + s_j)(s_i s_j)$ on $[a_1,b_k]$ and so
$$f:= s_i s_j - \frac{1}{m} \left( s_j^2 s_i + s_i^2 s_j\right) \geq 0 \text{ on } [a_1,b_k].$$
It is well-known that the quadratic module generated by
$\{ x- a_1, b_k - x \}$ is saturated (see, e.g., \cite[Cor. 11]{PR1}) and  hence $f \in M_S$.
Then $s_i s_j = f + \frac{1}{m} ( s_j^2 s_i + s_i^2 s_j) \in M_S$.

Finally, for $1 \leq i \leq k-1$, it's easy to check that $(x - a_1)(b_k - x) + s_i \geq 0$ on $[a_1,b_k]$ and
an argument similar to the previous argument shows that $ (x - a_1)(b_k - x) s_i \in M_S$.  
\end{proof}

Our goal in this section is to prove the following:

\begin{theorem} \label{striptheorem} Let $U$ and $S$ be as above and  $M$ the quadratic
module  in $\rr[x,y]$ generated by
$S$.  Then $M$ is saturated.  In other words, if $f(x,y) \in \rr[x,y]$ is non-negative on $U \times \rr$, then $f \in M$.
\end{theorem}

We begin with a proof for the case where $f$ is a polynomial in $x$ only.  
%This follows
%from a more general theorem of Scheiderer \cite[Cor. 4.4]{S_05}, however we give an
%elementary, constructive proof for our special case.

\begin{prop}  \label{dim1case} Suppose $U \sse \rr$ is compact with $S$ the natural choice of generators.  Then the quadratic module in $\rr[x]$ generated by $S$  is saturated.  
\end{prop}

\begin{proof}  Let $T$ be the preordering in $\rr[x]$ generated by $S$.   By Lemma \ref{dim1},
it is enough to prove that  $T$ is saturated.  We note that the proof of this is similar to the
proof of \cite[Thm. 2.2]{KM_MP}.    

Suppose $f \in \rr[x]$ and $f \geq 0$ on $U$, then
 we can factor $f$ in $\rr[x]$ into psd quadratics and linear polynomials.  Since psd implies sos in $\rr[x]$, it
is enough to prove the proposition for $f$ a product of linear polynomials.  We proceed by induction on $d = \deg f$.  The
$d = 0$ case is trivial.  So suppose $x - r$ is a factor of $f$ and write $f = (x - r) f_1$.  If $r \leq a_1$, then
$x - r \geq 0$ on $U$, hence $f_1 \geq 0$ on $U$ as well.  Then $f = [(x - a_1) + (a_1 - r)] f_1$ and we are done
since $(x-a_1) + (a_1 - r) \in T$, and $f \in T$ by induction.  The case of $r \geq b_k$ is similar.

Now suppose $a_i \leq r \leq b_{i+1}$ for some $i$.  Since $f$ changes sign at $r$, there must be another
root $s$ of $f$ with $a_i \leq s \leq b_{i+1}$.  Then $f = (x - r)(x - s) f_1$ with $f_1 \geq 0$ on $U$,
and, by \cite[Lemma 4]{BCJ}, $(x -r)(x - s)$ is in the preordering generated by $(x - b_i)(x - a_{i+1})$ and
hence in $T$.  Since $f_1 \in T$ by induction, we have $f \in T$ in this case as well.
\end{proof}

It follows immediately that Theorem \ref{striptheorem} is true
if $f$ is a polynomial in $x$ only.  So suppose we have $f \in \rr[x,y]$ such that
$f \geq 0$ on $U \times \rr$ and $\deg_y f \geq 1$.  Since $f$ is positive as $|y| \rightarrow \pm \infty$, it
follows that $f$ has even degree as a polynomial in $y$ and that the leading coefficient of 
$f$ as a polynomial in $y$ is non-negative on $U$.

Next we show that it is enough to prove Theorem \ref{striptheorem} for the
case where the leading coefficient of $f$ (as a polynomial in $y$) is positive on $U$.   The proof
is a straightforward
generalization of the proof of \cite[Lemma 2.1]{MM_strip}.

\begin{lemma}\label{Lem4} It is enough to prove Theorem \ref{striptheorem}
for $f \in \rr[x,y]$ such that the leading coefficient of $f$ as a
polynomial in $y$ is strictly positive on $U$.
\end{lemma}

\begin{proof}  Arguing exactly as in the proof of \cite[Lemma 2.1]{MM_strip},
we can reduce to showing that if
$h \in \rr[x]$ with $h \geq 0 \text{ on } U$,  and $h$ is $\pm$ a product of
linear factors $x - r$ with $r \in U$,
then for any $f \in \rr[x,y]$, $h f \in M$ implies $f \in M$.
The proof is by induction on $\deg h$.  If $\deg h = 0$, this is trivial, hence we assume $\deg h \geq 1$. 

For ease of exposition,  let $s_0 = 1, s_1 = x - a_1, s_2 = (b_1 - x)(a_2 - x) , \dots , s_k = (b_{k-1} - x)(a_k - x),
s_{k+1} = b_k - x$.  Since $hf \in M$, we have
\begin{equation} \label{lem4.1}
hf  =  \sig_0 s_0 + \sig_1 s_1 + \dots + \sig_{k+1} s_{k+1},
\end{equation}
where each $\sig_i \in \sum \rr[x,y]^2$.

Given $r \in U$ and suppose $x - r$ is  a factor of $h$.  There are several cases to consider.
\smallskip

\noindent
Case 1:  Suppose  $r$ is in the interior of $U$,
then since  $h$ does not change sign at $r$, it follows that
$(x - r)^2$ divides $h$.
Substituting $x = r$ into both sides of \eqref{lem4.1}, we have
$0 = \sum_{i=0}^{k+1} \sig_i(r,y) s_i(r)$.
Since  each $s_i(r)$ is positive, it follows that
$\sig_i(r,y) = 0$ for all $y \in \rr$.  Thus $\sig_i(r,y)$ is
identically zero, which implies that $x - r$ divides each
coefficient of $\sig_i(x,y)$, and consequently $x - r$ divides $\sig_i(x,y)$.
Since $\sig_i(x,y)$ is a sum of squares, it follows that $(x - r)^2$
divides $\sig_i(x,y)$.  Dividing both sides of \eqref{lem4.1} by $(x - r)^2$,
we are done by induction.

\medskip

\noindent
Case 2:  Suppose $s_1 = x - a_1$ or $s_{k+1} = x - b_k$ divides $h$.  We give the proof for $s_1$,
the proof for $s_{k+1}$ is the same.
If $ x - a_1$ divides $h$, substituting
$x = a_1$ into  \eqref{lem4.1}, we have
$0 = \sig_0(a_1,y)  + \sum_{i=2}^{k+1} \sig_i(a_1,y) s_i(a_1) $.  
Since $s_i(a_1) > 0$ for $2 \leq i \leq k+1$, arguing 
as in the first case, this implies that $(x - a_1)^2$ divides $\sig_i(x,y)$ for
$i = 2, \dots , k+1$.   Let $\tau_i(x,y) = \sig_i(x,y)/(x - a_1)^2 \in \sum \rr[x,y]^2$.   Dividing both sides of \eqref{lem4.1} by $x - a_1$, we
obtain
\begin{equation} \label{eqn2}
\frac{h}{x - a_1} f =  \tau_0  (x - a_1)  + \sig_1  + 
\tau_2 (x - a_1) s_2 + \dots +\tau_{k+1} (x - a_1) s_{k+1}
\end{equation}
By Lemma \ref{dim1},  $M$ is closed under
multiplication, hence $(x - a_1) s_i \in M$ for each $i$.  It follows that
 the right-hand side of \eqref{eqn2} is in
$M$ and we are done by induction.

\medskip

\noindent
Case 3:  Suppose neither Case 1 nor Case 2 applies, then $h$ contains a factor $x - a_i$ for
$2 \leq i \leq k$, or $x - b_i$ for $1 \leq i \leq k - 1$.  We give the proof for $x - a_i$, the
proof for $x - b_i$ is the same.  Since $h \geq 0$ on $U$ and  does
not change sign at any interior point of $U$, it follows that $h$ contains a factor $(x - a_i)^2$
or  a factor $(x - a_i)(b_i - x) = s_i$.   In the first case, applying the argument of Case 2
twice, we see that  $(x - a_i)^2$ must divide every term on the right-hand side of \eqref{lem4.1} and we
are done by induction. In the second case, we argue as in Case 2 to conclude that $s_i$ divides
every term on the right-hand side of \eqref{lem4.1} and we are again done by induction.

\end{proof}

\begin{lemma}\label{Lem6} We may assume that $f$ has finitely many zeros on $U \times \rr$.
\end{lemma}

\begin{proof} The proof is essentially the same as the proof of \cite[Lemma 2.2]{MM_strip}.
\end{proof}

\begin{lemma}\label{Lem7} Suppose $f = \sum_{i=0}^{2d} a_i(x) y^i$ is non-negative on $U \times \rr$,
$f$ has only finitely many zeros in $U \times \rr$,
and $a_{2d} > 0$ on $U$.  Then there exists $\ep(x) \in \rr[x]$, with $\ep(x) \geq 0$ on $U$, such that $f(x,y) \geq \ep(x)(1 + y^2)^d$ holds on $U \times \rr$, and for each $x \in U$, $\ep(x) = 0$ if and only if there exists $y \in \rr$ such that $f(x,y) = 0$.
\end{lemma}

\begin{proof}By \cite[Lemma 4.2]{MM_strip} and its proof, for $i = 1, \dots, k$, there exists a polynomial $\ep_i(x) \in \rr[x]$, with  $\ep_i(x) \geq 0 \text{ on } [a_i, b_i]$, such that $f(x,y) \geq \ep_i(x)(1 + y^2)^d$ holds on $[a_i, b_i] \times \rr$, $\ep_i(x) = 0$ for $x \in [a_i, b_i]$ if and only if there exists $y \in \rr$ such that $f(x,y) = 0$, and $\ep_i(x) \neq 0$ for $x \in \rr \backslash [a_i, b_i]$.

Dividing each $\ep_i$ by the maximum of $\{ \ep_i(x) \mid x \in U \}$ and $1$, we may assume that each 
$\ep_i(x) \leq 1$ on $U$.  Let $\ep(x) =  \left(\prod_{i=1}^k \ep_i(x)\right)^2$, then $\ep(x) \geq 0$ on $U$, and $$f(x,y) \geq \ep(x)(1 + y^2)^d$$ holds on $U \times \rr$.  For each $x \in U$, the polynomial $\ep(x) = 0$ if and only if some $\ep_i(x) = 0$, hence $\ep(x) = 0$ if and only if there exists $y \in \rr$ such that $f(x,y) = 0$.
\end{proof}

In \cite[Lemma 4.4]{MM_strip}, it is shown that if $f \in \rr[x,y]$ such that $f \geq 0$ on $[0, 1] \times \rr$ and the leading coefficient of $f$ is positive on the interval $[0,1]$, then for each $r \in [0, 1]$ there is a representation of $f$ involving the generators of the quadratic module and functions of the
form  $\sum g_i^2$, where each $g_i$ is a  polynomial in $y$ with coefficients analytic functions of $x$ in some neighborhood of $r$.  In our case, we need the same result with $[0,1]$ replaced by $U$.
  This follows immediately from the \cite[Lemma 4.4]{MM_strip} unless $r = a_i$ for $2 \leq i \leq k$ or $r = b_i$ for $1 \leq i \leq k - 1$; for the latter cases we need one extra step.

\begin{lemma}\label{Lem8}Suppose $f \in \rr[x,y]$ is non-negative on $U \times \rr$, and the leading coefficient of $f$ as a polynomial in $y$ is strictly positive on $U$.  Then:

\begin{enumerate}
    \item For each $r$ in the interior of $U$,  there exist $g_1, g_2$ polynomials in $y$ with coefficients analytic functions of $x$ in some open neighborhood $V(r)$ of $r$, such that $f = g_1^2 + g_2^2$ on $V(r) \times \rr$.
    \item There exist $g_l, h_l,$ with $l = 1, 2$, polynomials in $y$ with coefficients analytic functions of $x$ in some open neighborhood $V(a_1)$ of $a_1$ such that $f = \sum_{l=1}^2 g_l^2 + \sum_{l=1}^2 h_l^2 (x - a_1)$ on $V(a_1) \times \rr$.
    \item For $i = 1, \dots, k - 1$, there exist $g_l, h_l,$ with $l = 1, 2$, polynomials in $y$ with coefficients analytic functions of $x$ in some open neighborhood $V(b_i)$ of $b_i$ such that $f = \sum_{l=1}^2 g_l^2 + \sum_{l=1}^2 h_l^2(b_i - x)(a_{i+1} - x)$ on $V(b_i) \times \rr$.
   \item For $i = 1, \dots, k - 1$, there exist $g_l, h_l, l = 1, 2$, polynomials in $y$ with coefficients analytic functions of $x$ in some open neighborhood $V(a_{i+1})$ of $a_{i+1}$ such that $f = \sum_{l=1}^2 g_l^2 + \sum_{l=1}^2 h_l^2(b_i - x)(a_{i+1} - x)$ on $V(a_{i+1}) \times \rr$.
     \item There exist $g_l, h_l$,with $l = 1, 2$, polynomials in $y$ with coefficients analytic functions of $x$ in some open neighborhood $V(b_k)$ of $b_k$, such that $f = \sum_{l=1}^2 g_l^2 + \sum_{l=1}^2 h_l^2 (b_k - x)$ on $V(b_k) \times \rr$.
\end{enumerate}
\end{lemma}

\begin{proof} (1), (2) and (5) follow from \cite[Lemma 4.4]{MM_strip}, using a change of variables, if necessary.

For (3), if $x$ is sufficiently close to $b_i$, by \cite[Lemma 4.4]{MM_strip}, there exist $\varphi_l(x,y), \psi_l(x,y)$, $l = 1, 2$, polynomials in $y$ with coefficients analytic functions of $x$ in some open neighborhood $V(b_i)$ of $b_i$, such that $$f = \sum_{l=1}^2\varphi_l^2 + \sum_{l=1}^2 \psi_l^2 \ts (b_i - x).$$  We have
\begin{displaymath}
\begin{array}{lll}
f &=&\displaystyle \sum_{l=1}^2\varphi_l^2 + \sum_{l=1}^2 \frac{\psi_l^2}{(a_{i+1} - x)} \ts (b_i - x)(a_{i+1} - x)\\
      &=&\displaystyle \sum_{l=1}^2 \varphi_l^2 + \sum_{l=1}^2 \left(\frac{\psi_l}{\sqrt{a_{i+1} - x}}\right)^2 \ts (b_i - x)(a_{i+1} - x).\\
\end{array}
\end{displaymath}

As $\displaystyle \frac{1}{\sqrt{a_{i+1} - x}}$ is analytic for $x$ close to $b_i$, by taking $g_l = \varphi_l$ and \\ $h_l = \displaystyle \frac{\psi_l}{\sqrt{a_{i+1} - x}}$, we get the desired result.

A similar proof shows that (4) holds.
\end{proof}

We need the following version of the Weierstrass Approximation Theorem, which is an immediate generalization of \cite[Proposition 4.5]{MM_strip}

\begin{prop}\label{Prop9} Suppose $\phi, \psi: U \rightarrow \mathbb{R}$ are continuous functions, where $U \sse \rr$ is compact, $\phi(x) \leq \psi(x)$ for all $x \in U$, and $\phi(x) < \psi(x)$ for all but finitely many $x \in U$.  If $\phi$ and $\psi$ are analytic at each point $a \in U$ where $\phi(a) = \psi(a)$ then there exists a polynomial $p(x) \in \mathbb{R}[x]$ such that $\phi(x) \leq p(x) \leq \psi(x)$ holds for all $x \in U$.
\end{prop}
\medskip

We are now ready to prove Theorem \ref{striptheorem}.  For ease of exposition,
denote the natural choice of generators $S$ for $U$ by
$\{s_1, \dots, s_{k+1}\}$, i.e., $$s_1 = x - a_1, s_2 = (b_1 - x)(a_2 - x),
\dots, s_{k+1} = b_k - x.$$
Let $f(x,y) = \displaystyle \sum_{j=0}^{2d} a_j(x)y^j,$ where $d \geq 1$,
$a_{2d}(x) > 0$ on $U$, and $f(x,y)$ has only finitely many zeros in
$U \times \rr$.  By Lemma \ref{Lem7}, we have $\ep(x) \in \rr[x]$ such
that $\ep(x) \geq 0$ on $U$, $f(x,y) \geq \ep(x)(1 + y^2)^d$, and $\ep(x) = 0$
iff $y \in U$ such that $f(x,y) = 0$.  Let $f_1(x,y) :=
f(x,y) - \ep(x)(1 + y^2)^d$, then $f_1 \geq 0$ on $U \times \rr$.  Replacing
$\ep(x)$ by $\frac{\epsilon(x)}{N}, N > 1$, if necessary, we can assume
$f_1$ has degree $2d$ as a polynomial in $y$, and the leading coefficient of
$f_1$ is positive on $U$.

By Lemma \ref{Lem8}, for each $r \in U$, there exists an open neighborhood $V(r)$ of $r$ so that
\begin{equation}\label{eqn4}
f_1 =\displaystyle \sum_{j=1}^2 g_{0,j,r}(x,y)^2 + \sum_{j=1}^2 g_{1,j,r}(x,y)^2 \ts s_1 + \dots + \sum_{j=1}^2 g_{k+1,j,r}(x,y)^2 \ts s_{k+1}
\end{equation} on $V(r) \times \rr$, where $g_{i,j,r}(x,y)$ are polynomials in $y$ of degree $\leq d$ with coefficients analytic functions of $x$ in $V(r)$, for $i = 0,..., k + 1$ and $j = 1, 2$.  If $r$ is in the interior of $U$, note that $g_{i,j,r} = 0$ for $i \neq 0$.  If $r = a_1$, then $g_{i,j,r} = 0$ for $i \neq 1$, etc.

The rest of the proof follows along the lines of the proof of Theorem \ref{mar}.
Since $U $ is compact, there are finitely many $V(r_1),\dots, V(r_p)$ which
cover $U$ and, since $\ep(x)$ has only finitely many roots in $U$,  we choose the open cover so that no $V(r_l)$ contains more
than one root of $\ep(x)$, and no root is in more than one $V(r_l)$.
Let $1 = \nu_1 +...+ \nu_p$  be a partition of unity corresponding to the
open cover of $U$, and note that by construction, if a root $u$ of
$\ep(x)$ is in $V(r_l)$, then $\nu_l(x) = 1$ for $x$ close to $u$.
Since $U $ is compact, there are finitely many $V(r_1),\dots, V(r_p)$ which cover $U$.

Define $\varphi_{i,j,l}$, polynomials in $y$ with coefficients functions
of $x$ as follows:  The coefficient of $y^q$ in $\varphi_{i,j,l}$ is
$\sqrt{\nu_l(x)}$ times the coefficient of $y^q$ in $g_{i,j,r_l}$.  Then
we have
\begin{equation}
f_1 = \displaystyle \sum_{l=1}^p \nu_l f_1 = \sum_{l=1}^p \left(\sum_{j=1}^2\varphi_{0,j,l}^2 + \sum_{j=1}^2\varphi_{1,j,l}^2 \ts s_1 + \dots + \sum_{j=1}^2\varphi_{k+1,j,l}^2 \ts s_{k+1}\right)
\end{equation} on $U \times \rr$.

We approximate the coefficients of the $\varphi_{i,j,l}$'s by
polynomials, using Proposition \ref{Prop9}.  Fix $\varphi_{i,j,l}$ and a
coefficient $u(x)$.
Define $\phi, \psi: U \rightarrow \rr$ by $\phi(x) =
u(x) - \displaystyle \frac25 \ep(x)$, and $\psi(x) = u(x) + \displaystyle
\frac25 \ep(x)$.  Then by our construction, $\phi(x)$ and $\psi(x)$
satisfy all of the conditions of Proposition \ref{Prop9}, and so
there exists $w \in \rr[x]$ such that
\begin{equation}\label{eqn5}
u(x) - \displaystyle \frac{2}{5} \ep(x) \leq w(x) \leq u(x) + \frac{2}{5} \ep(x), \text{ for each } x \in U.\end{equation}

Now we use these $w(x)$'s to define, for each triple $i, j, l$, a polynomial $h_{i,j,l}$, where $\deg_y h_{i,j,l} = \deg_y \varphi_{i,j,l}$, and if $u(x)$ is the coefficient of $y$ in $\varphi$, and $w(x)$ is the coefficient of $y$ in $h$, then \eqref{eqn5} holds.  Finally, let

$$ h_l(x,y) := \sum_{j=1}^2 h_{0,j,l}(x,y)^2 + \sum_{j=1}^2 h_{1,j,l}(x,y)^2 \ts s_1 + \dots + \sum_{j=1}^2 h_{k+1,j,l}(x,y)^2 \ts s_{k+1}$$

We have polynomials $h_l$ and $\delta \in \rr[x,y]$ such that
$$f_1 = \displaystyle \left(\sum_{l=1}^p h_l(x,y)\right) + \delta(x,y),$$  where $\delta(x,y) = \sum_{i=0}^{2d} c_i(x)y^i$ and  $|c_i(x)| \leq \frac{2}{5}
\ep(x)$ on $U$, for all $i$.

This yields $f(x,y) = f_1(x,y) + \epsilon(x)(1 + y^2)^d =
\sum_{l=1}^p h_l(x,y) + t_1(x,y) + t_2(x,y)$, where

$t_1(x,y) := \displaystyle \frac{2}{5}\epsilon(x)(2 + y + 3y^2 + y^3 +3y^4 + ... + y^{2d-1} + 2y^{2d}) + \sum_{i=0}^{2d} c_i(x)y^i$,\\
$t_2(x,y) := \displaystyle \epsilon(x)[(1 + y^2)^d - \frac{2}{5}(2 + y + 3y^2 + y^3 + 3y^4 + ... + y^{2d-1} + 2y^{2d})]$.\\

We have $ \sum_{l=1}^p h_l(x,y) \in T$  and we can prove that
$t_1, t_2 \in T$ exactly as in \cite{MM_strip}.
Therefore $f(x,y) \in T$.  This completes the proof of Theorem \ref{striptheorem}.

\medskip

Suppose $\tilde U \sse \rr$ is a non-compact basic closed semialgebraic set.  An obvious question
to ask is what happens if we replace $U \times \rr$ by $U \times \tilde U$?  First we note that by
\cite[Thm. 2]{PR1}, if $S \sse \rr[x,y]$ such that  $K_S = U \times \rr^+$, then $M_S$ cannot be
saturated, regardless of the choice of generators $S$.   Furthermore, if $S \sse \rr[x]$ generates
$\tilde U$ as a semialgebraic set in $\rr$, then $T_S$ is saturated iff $S$ contains the natural choice
of generators \cite[Thm. 2.2]{KM_MP}.  This means that the best theorem we could hope for
is the following:  Let $S_1 \sse \rr[x]$ be the natural choice of generators for $U$ and $S_2 \sse \rr[y]$
the natural choice of generators for $\tilde U$, then the preordering in $\rr[x,y]$ generated by
$S_1 \cup S_2$ is saturated.  We have the following partial result, which is \cite[Cor. 11]{P}:

\begin{theorem}  Let $U$, $\tilde U$, $S_1$, and $S_2$ be as above  and $T$ the preordering
in $\rr[x,y]$ generated by $S_1 \cup S_2$.  If 
$f = \sum_{i=0}^d a_i(x) y^i \sse \rr[x,y]$ such that $f > 0$ on $U \times \tilde U$ and
$a_d > 0$ on $U$, then $f \in T$.
\end{theorem}

\medskip
\noindent
{\bf Question}:  Is the above theorem true without one or both of the assumptions on $f$?

\section{Half-strips and further examples}

In this section we look at non-compact basic closed semialgebraic
subsets of a strip $[a,b] \times \rr$   which are bounded as
$y \rightarrow -\infty$.  We
refer to such a set as a {\it half-strip} in $\rr^2$.  We give a
representation theorem for the half-strip 
$\{ (x,y) \in \rr^2 \mid x \in U,  y \geq q(x) \}$,  where
$U \sse \rr$ is compact and $q(x) \in \rr[x]$ with $q(x) \geq 0$ on $U$.
This follows from
Theorem \ref{striptheorem} by an
elementary argument. We give a few other examples of saturated preorderings in
the half-strip case as well as a family of negative examples.  Finally, we use
Theorem \ref{striptheorem} to 
give an example of a non-compact surface in $\rr^3$ for which
the corresponding preordering is saturated.

\begin{remark} Suppose $U \sse \rr$ is compact and $S$ the natural choice of
generators for $U$.  We saw in the previous section that in $\rr[x]$, 
the preordering generated by $S$ and the quadratic module generated by $S$ are the same
and hence the same is true in $\rr[x,y]$.  
However, in \cite[Thm. 2]{PR}, it is shown that if $S$ any set of generators
in $\rr[x]$ for $[0,1]$, then the quadratic module generated by $S$ and $y$
is not saturated.   Hence in the half-strip case, our representation theorems
will hold only for  preorderings and not quadratic modules as in the strip case.
\end{remark}

\begin{theorem} \label{halfstrip} Given compact $U \sse \rr$ with natural choice of
generators $\{ s_1, \dots , s_k \}$ and $q(x) \in \rr[x]$ with $q(x) \geq 0$ on $U$, set
$S = \{s_1, \dots , s_k, y - q(x)\}$ and let $K$  be the half-strip $K_S$.  If $T$ is
the preordering in $\rr[x,y]$ generated by $S$, then $T$ is saturated.
\end{theorem}

\begin{proof}  We first claim that it is enough to prove the theorem
for $q(x) = 0$, i.e., the case where the semialgebraic set is $U \times \rr^+$ 
with generators $\{ s_1, \dots , s_k, y \}$.
Suppose that the preordering $W \sse \rr[u,v]$ generated by $\{s_1(u), \dots , s_k(u), v \}$ is saturated
and we have $f(x,y)  = \sum_{i=0}^{k} a_i(x) y^i$ is non-negative on $K$.
 Define
$g$ in $\rr[u,v]$ by $g(u,v) := \sum a_i(u)(q(u) + v)^j$.  Then $f(x,y)
\geq 0$ on $K$ implies $g(u,v) \geq 0$ on $U \times \rr^+$.  Hence $g \in W$.
Substituting $u = x, v = y - q(x)$ in a representation of $g$ in $W$, we obtain 
a representation of $f$ in $T$.

We are reduced to proving the theorem for $S = \{ s_1, \dots , s_k, y \}$.  If
$f(x,y) \geq 0$ on $U \times \rr^+$, then $f(x,y^2) \geq 0$ on
$U \times \rr$.   Then
by Theorem \ref{striptheorem}, we can write $f(x,y^2)$ as a sum of
terms of the form $(\sum_{i=1}^m h_i(x,y)^2 ) s_i$ (where we set $s_0 = 1$).

We have
$$\sum h_i(x,y)^2 = \frac{1}{2} \sum h_i(x,y)^2 + \frac{1}{2} \sum h_i(x,-y)^2.$$ 
Using the standard identity
$$ \frac12 \left(\sum_i a_iy^i\right)^2 +  \frac12 \left(\sum_i a_i(-y)^i\right)^2 =  \left(\sum_j a_{2j}y^{2j}\right)^2 + \left(\sum_j a_{2j+1} y^{2j}\right)^2 \cdot y^2$$
we have that $f(x,y^2)$ can be written as a sum of polynomials of the form
$$\left(\sum_{i=1}^l \sig_i(x,y^2)^2 + \tau_i(x,y^2)^2 \cdot y^2\right) s_i.$$
Replacing $y^2$ by $y$ yields a representation of $f(x,y)$ in $T$.
\end{proof}

 Combining Theorem \ref{halfstrip} with a substitution technique from
 work of Scheiderer \cite{S}, we can obtain more examples of half-strips
 for which the corresponding preordering is saturated.

 \begin{example}  Let $S = \{x - x^2, xy - 1\}$ so that $K_S$ is the upper
half of $\rr^+$
cut by $xy = 1$.   We claim that
$T_S$ is saturated.  

%%% Half-strip $xy - 1$ Figure(1figure)
\begin{figure}[ht]
\centering
\includegraphics[scale=0.6]{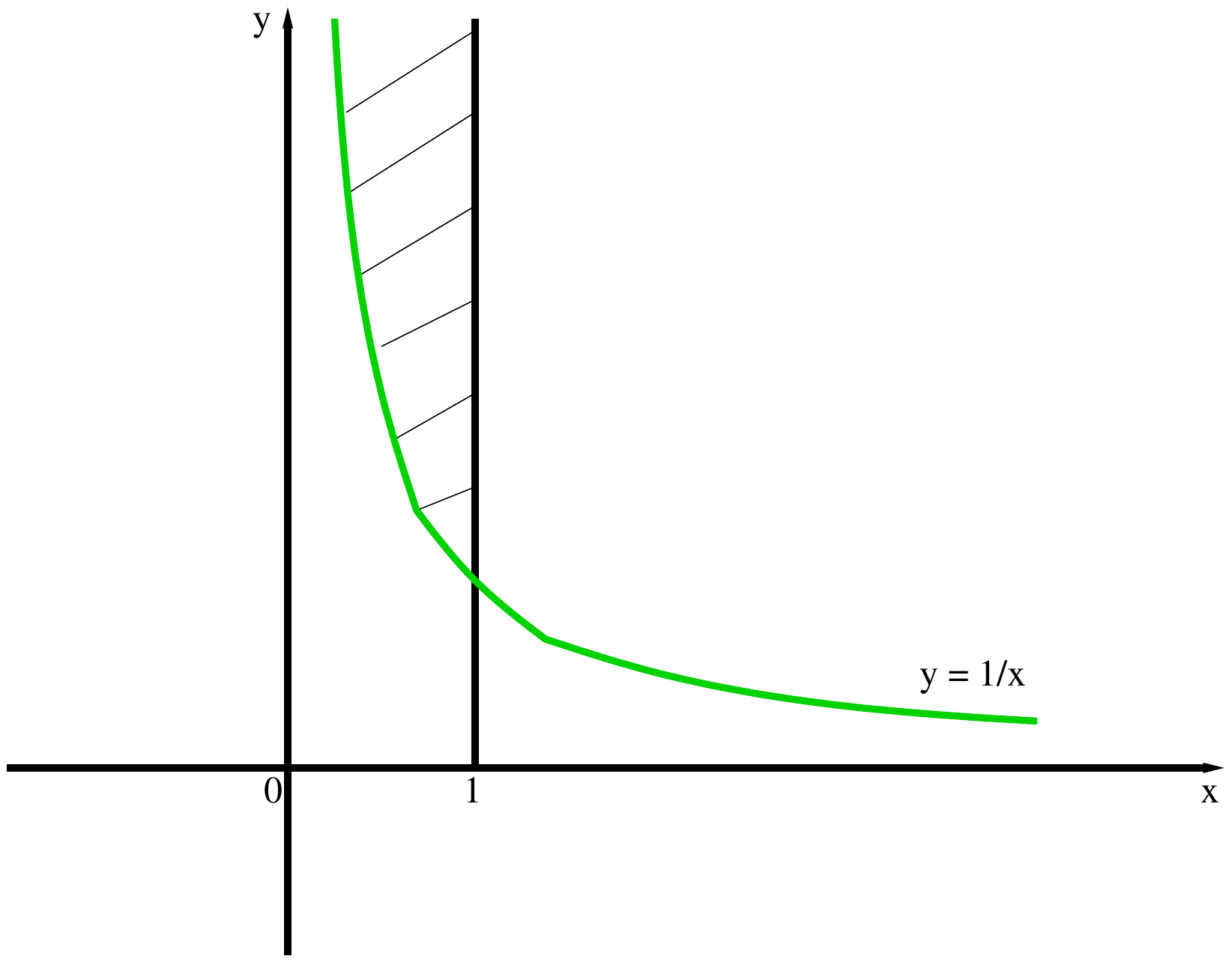}
\caption{$\rr^+$  cut by $xy = 1$}\label{halfstripxy}
\end{figure}

 Suppose $f(x,y) \geq 0$ on $K_S$.  Pick an integer $n \geq 0$ large enough so that $x^{2n}f \in \rr[x, xy]$.  Define $g$ in $\rr[u,v]$ by $g(u,v) := u^{2n}f(u,\frac{v}{u})$, i.e., $g(x, xy) = x^{2n}f(x,y)$ .  As $f(x,y) \geq 0$ on $K_S$, $g(u,v) \geq 0$ on
 $[0,1] \times [1,\infty)$, hence by Theorem \ref{halfstrip} there exist sums of squares $\sig_0, \sig_1, \sig_2, \sig_3 \in
 \rr[u,v]$ such that
$$g(u,v) = \sig_0 + \sig_1 (u - u^2) + \sig_2 (v-1) + \sig_3 (u - u^2)(v -1).$$
Then  $x^{2n} f(x,y) = $
\begin{equation} \label{sub}
 \sig_0(x,xy) + \sig_1(x,xy)(x - x^2) + \sig_2(x,xy)(xy -1)+ \sig_3(x,y)(x - x^2)(xy - 1).
\end{equation}

Define $s_m :=\displaystyle \frac{\sig_m}{x^{2n}}, m = 0,\dots,3$. As $x^{2n}$ divides each of the sums on the right hand side of
(\ref{sub}), the $s_m$ are sos in $\rr[x,y]$.  Thus $f$ can be written as $$f(x,y ) = s_0(x,y) + s_1(x,y)(xy - 1) + (s_2(x,y) + s_3(x,y)(xy - 1))(x- x^2)$$  Hence $f \in T_S$.
\end{example}

Next we give an example of $S \sse \rr[x,y,z]$ such that $K_S$ is non-compact of dimension 2, and $T_S$ is saturated.

\begin{example} Suppose $S = \{1 - x^2, z - x^2, x^2 - z\}$ so that $ K_S $ 
$\{(x,y,z) \mid -1 \leq x \leq 1, z = x^2\}.$    We claim $T_S$ is saturated.

%%% HalfStrip in R^3 Example Figure(1figure) 
%\begin{figure}[ht]
%\centering
%\includegraphics[scale=0.4]{HalfStripR3.png}
%\caption{Half-strip in $\rr^3$}
%\end{figure}

Given $f(x,y,z) \geq 0$ on $K_S$, write $f = \sum g_i(x,y) z^i = \sum g_i(x,y)(z^i - x^{2i}) + \sum g_i(x,y) x^{2i},$ where $g_i(x,y) \in \rr[x,y]$.  Then $\sum g_i(x,y)(z^i - x^{2i})$ is in the ideal generated by $z - x^2$ and hence in $T_S$.  
Let $g(x,y) = \sum g_i(x,y) x^{2i} = f(x,y,x^2)$.  Since $f(x,y,z) \geq 0$ on $K$, this implies that $g(x,y) \geq 0$ on $[-1, 1] \times \rr$.  By Theorem \ref{mar}, we have $g(x,y) = \sig(x,y) + \tau(x,y)(1 - x^2)$, where $\sig, \tau \in \sum \rr[x,y]^2$.  Thus $f \in T_S$.
\end{example}

We end with  a family of examples of half-strips for which no corresponding finitely generated
preordering is saturated.  This is a generalization of an example due to T.~Netzer, see
\cite[Lemma 7.4]{CKS}.

\begin{prop} Suppose $m \in \nn$ is even and $q(x) \in \rr[x]$ with
$\deg q$ odd and $q(x) \geq 0$ on $[0,1]$. Let
$K = \{ (x,y) \in \rr^2 \mid 
0 \leq x \leq 1,y^m \geq q(x), y \geq 0 \}$.  Then is no finite set of generators $S \sse \rr[x,y]$
with $K_S = K$ such that $T_S$ is saturated. 
\end{prop}

\begin{proof}
Suppose $S = \{ g_1,\dots, g_s \} \sse \rr[x,y]$ is such that $K_S = K$ and
$T_S$ is saturated.  For $c \in [0,1]$, let $T_c$ be the preordering in $\rr[x]$ generated by
$\{ g_1(c,y),\dots, g_s(c,y)\}$, then $T$ saturated implies that $T_c$ is saturated.  Since
$\{g_1(c,y) \geq 0,\dots,g_s(c,y) \geq 0\} = [q(c)^\frac{1}{m}, \infty)$, by Theorem 2.1 and 2.2 in \cite{KM_MP},
$y - q(c)^{\frac{1}{m}}$ must be among the $g_i(c,y)$ up to a constant factor.  Without lost of generality, we can assume $$g_1(c,y) = r(c)\left(y - q(c)^{\frac{1}{m}}\right),$$ for infinitely many $c \in [0,1]$ and some positive function $r$.  Let $d$ be the degree of $g_1(x,y)$ in $y$, and write $g_1(x,y) = \sum_{i=0}^d a_i(x) y^i$ with $a_i(x) \in \rr[x]$.  Then $$g_1(c,y) =  r(c)\left(y - q(c)^{\frac{1}{m}}\right) = a_0(c) + a_1(c)y +\dots+ a_d(c)y^d$$  for infinitely many $c \in [0,1]$.  Comparing coefficients, this implies $a_0(c) = -r(c)q(c)^{\frac{1}{m}}$ and $a_1(c) = r(c)$ for infinitely many $c \in [0,1]$.  Hence, since $a_0$, $a_1$ are polynomials, $a_0(x)^m = a_1(x)^m q(x) \in \rr[x]$.  But this is a contradiction, since the degree of the left-hand side is $m \cdot$ deg $a_0(x)$ while the degree of the right-hand side is $m \cdot$ deg $a_1(x) +$ deg $f(x)$, which implies that one is even and one is odd.

\end{proof}

\begin{example}  Suppose $S = \{ x - x^2, y^2 -x , y \}$, so that $K_S$ is the half-strip $[0,1] \times
\rr^+$ cut by the parabola $y^2 = x$.  Then, by the previous proposition, no finitely generated
preordering corresponding to $K_S$ is saturated.
\end{example}  

%%% Half-strip Negative Figure(1figure)
\begin{figure}[ht]
\centering
\includegraphics[scale=0.6]{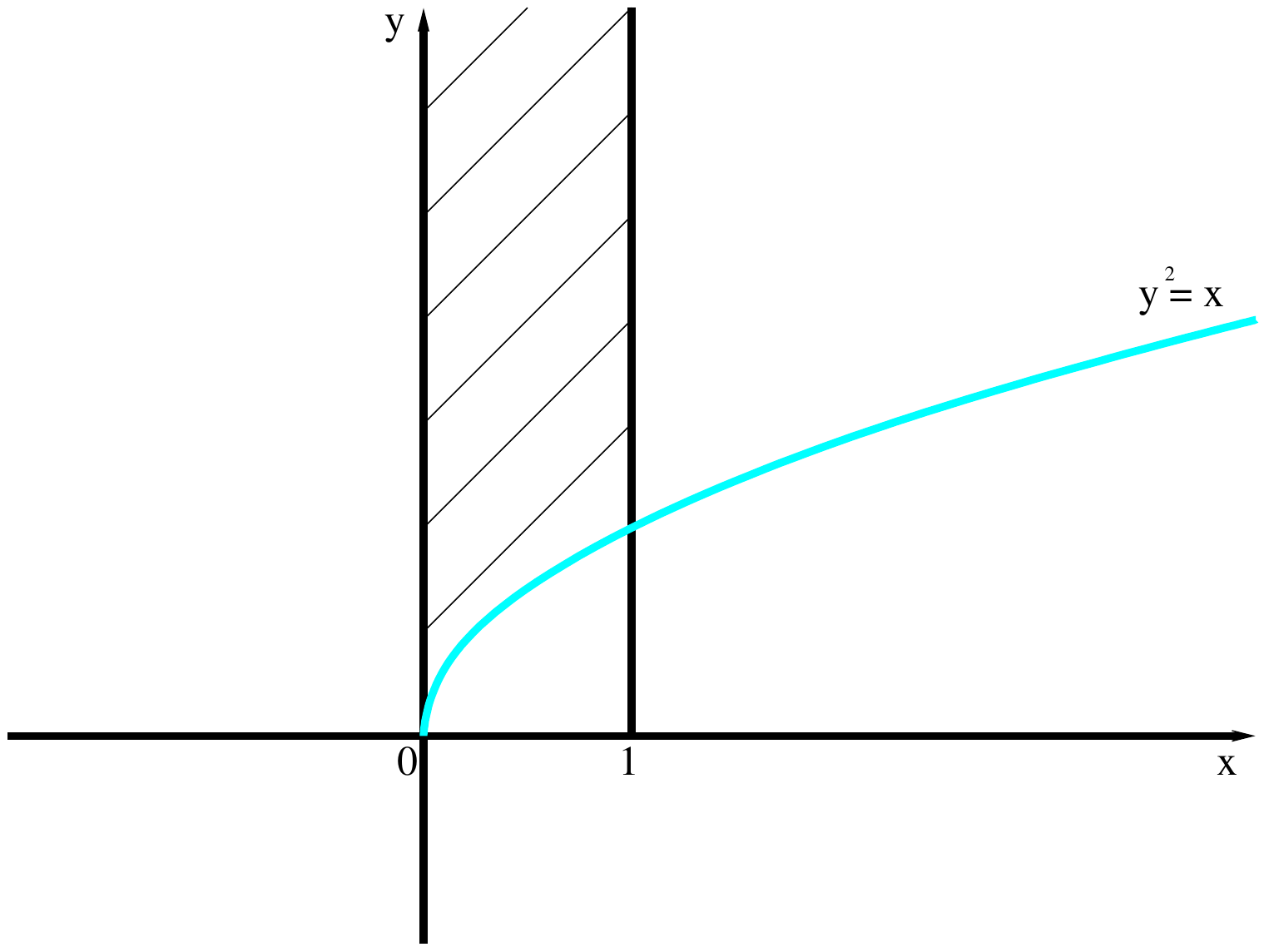}
\caption{Half-strip cut by $y^2 = x$}\label{halfstripxyneg}
\end{figure}

\bibliographystyle{amsplain}

\bibliography{StripRefs}

\end{document}